\documentclass[11pt]{amsart}
\usepackage{amsmath}
\usepackage{amssymb}
\usepackage{amsthm}
\usepackage{mathrsfs}
\usepackage{enumerate}
\usepackage{bm}
\theoremstyle{plain}
\newtheorem{Def}{Definition}[section]
\newtheorem{Thm}{Theorem}[section]
\newtheorem{Prop}{Proposition}[section]
\newtheorem{Lem}{Lemma}[section]

\newtheorem{Cor}{Corollary}[section]

\newtheorem{Rmk}{Remark}[section]

\makeatletter
    
    \@addtoreset{equation}{section}
 \makeatother
\begin{document}
\title[Lagrangian submanifolds in  bidisks
]{On a family of Lagrangian submanifolds in bidisks and Lagrangian Hofer metric}
\author{YUSUKE MASATANI}

\address{Graduate School of Mathematics, Nagoya University, Nagoya, Japan}

\email{m10039c@math.nagoya-u.ac.jp}


\begin{abstract}
We construct a family of uncountably many Lagrangian submanifolds 
in the standard bidisks such that the Lagrangian Hofer diameter associated to 
each Lagrangian submanifold is unbounded.
 We also prove a certain inequality of the Lagrangian Hofer metric which 
is of the same type as S. Seyfaddini's for the case of the real form of the 
complex $n$-ball.\end{abstract}

\maketitle

\section{Introduction}
For a symplectic manifold $(M,\omega)$, we denote by ${\rm Ham}_c(M, \omega)$ the group of all compactly supported Hamiltonian diffeomorphisms on  $(M,\omega)$. For a Lagrangian submanifold $L$ of $(M,\omega)$, $\mathcal{L}(L)$ denotes the set of Lagrangian submanifolds which are Hamiltonian isotopic to $L$. The {\it Lagrangian Hofer pseudo-metric} $d$ on $\mathcal{L}(L)$ is defined by using the {\it Hofer norm} $\|\cdot \|$, which is introduced in \cite{Ho90}, as follows. 
 
\[d(L_0, L_1) := \inf \{ ~\| \phi \| \mid~ \phi (L_0) = L_1 ,~\phi \in {\rm Ham}_{c}(M, \omega) \}.\]
The Hofer norm $\| \phi \|$ is defined by
 
\[\|\phi\| := \inf \int_{0}^{1}\left( \max_{p\in M}H(t,p)-\min_{p\in M}H(t,p) \right) dt,\]
where the infimum runs over all compactly supported Hamiltonians $H \in C^{\infty}_{c}([0,1]\times M)$ having time-one map $\phi^{1}_{H}$ equal to $\phi$. 

 Chekanov showed in \cite{Ch00} that this pseudo-metric $d$ is non-degenerate  for any closed and connected Lagrangian submanifolds in tame symplectic manifolds. Although our Lagrangian submanifolds are not closed, the same proof as Chekanov's yields that $d$ is also non-degenerate for our cases below. 
 
 In \cite{Kh09}, Khanevsky proved unboundedness of this metric when the ambient space $M$ is an open unit disk $B^{2}:=\{z \in \mathbb{C}\mid |z| <1\} \subset\mathbb{C}$ and the Lagrangian submanifold $L$ is the real form $Re(B^{2}):=\{z\in B^{2} \mid \mathrm{Im}~z =0\}$ of the open unit disk. Seyfaddini generalized Khanevsky's unboundedness result to the case of higher dimensional open unit ball $B^{2n}$ in \cite{Se13}.
  
 In this paper, by adopting Seyfaddini's technique, we prove unboundedness of metric spaces $\mathcal{L}(L)$ for a certain continuous family of non-compact Lagrangian submanifolds in bi-disks, which are mutually non-Hamiltonian isotopic.

 \subsection{Main Result}
Let $B^{2}( r ) \subset {\mathbb C}$ be the open disk of radius $r > 0$ equipped with a symplectic structure $2 \omega_{0}$, where $\omega_{0}$ is the standard symplectic structure on $\mathbb{C}$ so that ${\rm vol}(D(r))=2\pi r^2$. We denote by $B^{2}$ the open unit disk $B^{2}(1)$. We put $(B^{2} \times B^{2} , \bar{\omega}_{0}) := (B^{2}(1) \times B^{2}(1), 2\omega_{0} \oplus 2\omega_{0})$ and define Lagrangian submanifolds $L_{\delta}$ by \[L_\delta := T_\delta \times Re(B^{2}) \subset B^{2} \times B^{2}\] for each $1 / 2 < \delta \le 1$. Here \[T_\delta := \{{| z_1 |}^2 = 1 /(2 \delta) \} \subset B^{2}\] and $Re(B^{2})$ is the real form of  $B^{2}$. 

We study the Lagrangian Hofer metric spaces $(\mathcal{L}(L_{\delta}), d)$ in this paper. We obtain the following:
\begin{Thm}\label{main1}
For any $1/2 < \delta \le 1$, $(\mathcal{L}(L_\delta), d)$ has an infinite diameter.
\end{Thm}

  In addition to unboundedness, we prove the following inequality for a subfamily of $\{L_{\delta}\}$.
  
\begin{Thm} \label{main2}
 For any $(2 + \sqrt{3}) /4 < \delta \le 1$, there exists a map $\Phi_\delta :C^{\infty}_c((0,1)) \to \mathcal{L}(L_\delta)$ such that 
 \[\frac{\|f - g \|_{\infty}- D_{\delta}}{C_{\delta}} \le d(\Phi_\delta (f), \Phi_\delta(g)) \le \|f - g \|,\]
where $C_{\delta}$ and $D_{\delta}$ denote positive constants.
\end{Thm}
 
 In this statement, $C^{\infty}_c((0,1))$ denotes the space of compactly supported smooth functions on an open interval $(0,1)$ and the two norms on $C^{\infty}_c((0,1))$ is defined by
\[\|f\|_{\infty} := \max_{x\in(0,1)} |f(x)|,\] and
\[\|f\| := \max_{x\in(0,1)} f(x) -  \min_{x\in(0,1)} f(x).\]
These norms are equivalent. We note that $\|f\|_{\infty}=\|f\|$ for any non-negative functions $f\ge0$.

\begin{Rmk}
\normalfont
\begin{enumerate}[$(1)$]

\item In \cite{Se13}, Seyfaddini proved the same type inequality as in Theorem \ref{main2} for the case of the real form $Re(B^{2n})$. To prove the inequality, he used a family of quasi-morphisms on ${\rm Ham_{c}}(B^{2n})$ which were constructed as pullbacks of the {\it single} Calabi quasi-morphism on ${\rm Ham_{c}}(\mathbb{C}P^{n})$ in \cite{EP03} via the same family of conformally symplectic embeddings in \cite{BEP04}.
\item On the other hand, to prove Theorem \ref{main2}, we use pullbacks of the {\it family} of Calabi quasi-morphisms on  ${\rm Ham_{c}}(S^{2}\times S^{2})$ constructed by Fukaya-Oh-Ohta-Ono in \cite{FOOO11}.
\item As for the condition on $\delta$ in Theorem \ref{main2}, see Remark \ref{delta_condition}.
\end{enumerate}
\end{Rmk}

 \subsection{Acknowledgement}
I am deeply grateful to my supervisor, Professor Hiroshi Ohta, for his support and valuable advice.


\section{Calabi quasi-morphisms on ${\rm Ham}_{c}(B^{2}\times B^{2}, \bar{\omega}_{0})$}
In \cite{BEP04}, Biran-Entov-Polterovich used a family of conformally symplectic embeddings to obtain a continuum of linearly independent Calabi quasi-morphisms on ${\rm Ham}_{c}(B^{n}, \omega_{0})$ as their pullbacks of a quasi-morphisim on ${\rm Ham}(\mathbb{C}P^n, \omega_{FS})$. In \cite{Se13}, Seyfaddini used the same family of conformally symplectic embeddings and constructed a family of quasi-morphisms on ${\rm Ham_{c}}(B^{2n})$ to prove unboundedness of $\mathcal{L}(Re(B^{2n}), d)$. 

 In this section, we also construct quasi-morphisms on ${\rm Ham}_{c}(B^{2}\times B^{2})$ associated with Fukaya-Oh-Ohta-Ono's symplectic quasi-morphisms $\mu^{\mathfrak{b}(\tau)}_{e_{\tau}}$ as in \cite{Se13}.

\subsection{Calabi quasi-morphisms and symplectic quasi-states}\label{q_mor and q_state}

 Entov and Polterovich developed a way to construct {\it Calabi quasi-morphisms} and {\it symplectic quasi-states} for some closed symplectic manifold $(M, \omega)$ in a series of papers \cite{EP03, EP06, EP09}. In this section, we briefly recall several terminologies and a generalization of their construction.

 A {\it quasi-morphism} on a group $G$ is a function $\mu : G \to \mathbb{R}$ which satisfies the following property: there exists a constant $D \ge 0$ such that
\[|\mu(g_{1} g_{2}) -\mu(g_{1})-\mu(g_{2})| \le D~~~~{\rm for~all}~g_{1}, g_{2} \in G.\]
The smallest number of such $D$ is called the {\it defect} of $\mu$ and we denote by $D_{\mu}$. A quasi-morphism $\mu$ is called {\it homogeneous} if $\mu(g^{m})=m\mu(g)$ for all $m\in \mathbb{Z}$.
 
 For any proper open subset $U \subset M$, the subgroup ${\rm Ham}_{U}(M,\omega)$ is defined as the set which consists of all elements $\phi \in  {\rm Ham}(M, \omega)$ generated by a time-dependent Hamiltonian $H_{t} \in C^{\infty}(M)$ supported in $U$. We denote by $\widetilde{\rm Ham}_{U}(M, \omega)$ the universal covering space of ${\rm Ham}_{U}(M, \omega)$. The Calabi morphism $\widetilde{{\rm Cal}}_{U}: \widetilde{\rm Ham}_{U}(M^{2n}, \omega) \to \mathbb {R}$ is defined by
\[\widetilde{{\rm Cal}}_{U} (\tilde{\phi}_{H}):= \int_{0}^{1}dt \int_{M}H_{t} \omega^{n},\]
where ${\phi}^{1}_{H} \in {\rm Ham}_{U}(M, \omega)$ and $\tilde{\phi}_{H}$ is the homotopy class of the Hamiltonian path $\{\phi^{t}_{H}\}_{t\in[0,1]}$ with fixed endpoints. If $\omega$ is exact on $U$, $\widetilde{{\rm Cal}}_{U}$ descends to ${\rm Cal}_{U}: {\rm Ham}_{U}(M, \omega) \to \mathbb {R}$.

 A subset $X \subset M$ is called {\it displaceable} if there exists a $\phi \in {\rm Ham}(M,\omega)$ such that $\phi(X) \cap \overline{X} = \emptyset$.
\begin{Def}[\cite{EP03}]
\normalfont
A function $\mu:  \widetilde{\rm Ham}(M, \omega) \to \mathbb {R}$ is called a homogeneous Calabi quasi-morphism if $\mu$ is homogeneous quasi-morphism and satisfies
\begin{itemize}
 \item (Calabi property) If $\tilde{\phi} \in \widetilde{\rm Ham}_{U}(M, \omega)$ and $U$ is a displaceable open subset of $M$, then
\begin{equation}
\mu(\tilde{\phi}) = \widetilde{{\rm Cal}}_{U}(\tilde{\phi}),
\end{equation}
where we regard $\tilde{\phi}$ as an element in $\widetilde{\rm Ham}(M, \omega)$.
\end{itemize}
\end{Def}
 
 For each non-zero element of quantum (co)homology $a \in QH(M)$, the {\it spectral invariant} $\rho(~\cdot~ ; a) : C^{\infty}([0,1]\times M) \to \mathbb{R}$ is defined in terms of Hamiltonian Floer theory (see \cite{Oh97}, \cite{Sc00}, \cite{Vi92} for the earlier constructions and \cite{Oh05} for the general non-exact case).
 
  In \cite{FOOO11}, Fukaya-Oh-Ohta-Ono deformed spectral invariants and obtained $\rho^{\mathfrak b}(~\cdot~; a)$ by using an even degree cocycle ${\mathfrak b} \in H^{even}(M, \Lambda_{0})$, where $a$ is an element of {\it bulk-deformed quantum cohomology} $QH_{\mathfrak b}(M, \Lambda)$ (see also \cite{Us11} for a similar deformation of spectral invariants). Here coefficient ring $\Lambda_{0}$, which is called {\it universal Novikov ring}, and its  quotient field $\Lambda$ are defined by
\begin{equation*}
\Lambda_0 : = \Biggl \{ \sum_{i = 0}^{\infty} a_i T^{\lambda_i} \biggl |~  a_i \in \mathbb{C}, ~\lambda_i  \in \mathbb{R}_{\ge 0}, ~ \lim_{i \to \infty} \lambda_i = + \infty \Biggr \}~,
\end{equation*}
\begin{equation*}
\Lambda: = \Biggl \{ \sum_{i = 0}^{\infty} a_i T^{\lambda_i} \biggl |~  a_i \in \mathbb{C}, ~\lambda_i  \in \mathbb{R}, ~ \lim_{i \to \infty} \lambda_i = + \infty \Biggr \} \cong \Lambda_{0}[T^{-1}]~.
\end{equation*}

  Every element $\tilde{\phi} \in \widetilde{\rm Ham}(M, \omega)$ is generated by some time-dependent Hamiltonian $H$ which is {\it normalized} in the sense $\int_{M}H_{t}\omega^{n}=0$ for any $t\in[0,1]$. The spectral invariant $\rho^{\mathfrak b}(~\cdot~; a)$ has the homotopy invariance property: if $F, G$ are normalized Hamiltonians and $\tilde{\phi}_{F}=\tilde{\phi}_{G}$, then $\rho^{\mathfrak b}(F;a)=\rho^{\mathfrak b}(G;a)$ (see Theorem 7.7 in \cite{FOOO11}). 
 Hence, the spectral invariant descends to $\rho^{\mathfrak b}(~\cdot~;a):\widetilde{\rm Ham}(M, \omega) \to \mathbb{R}$ as follows:
 \[\rho^{\mathfrak b}(\tilde{\phi}_{H};a) := \rho^{\mathfrak b}(\underline{H};a)~~~~~{\rm~for~any}~H\in C^{\infty}([0,1]\times M),\]
where we denote by $\underline{H}$ the normalization of $H$:
\[\underline{H}_{t}:=H_{t} - \frac{1}{{\rm vol}(M)}\int_{M^{2n}} H_{t} ~\omega^{n},~~~~~{\rm vol}(M):=\int_{M^{2n}}\omega^{n}.\]
 
 By using this (bulk-deformed) spectral invariant $\rho^{\mathfrak b}(~\cdot~; a)$, as in a series of papers \cite{EP03, EP06, EP09}, they constructed a function $\mu^{\mathfrak b}_{e}: \widetilde{\rm Ham}(M, \omega) \to \mathbb{R}$ by
\[\mu^{\mathfrak b}_{e}(\tilde{\phi}):= {\rm vol}(M) \lim_{m \to +\infty}\frac{\rho^{\mathfrak b}({\tilde{\phi}}^{m}; e)}{m},\] 
where  $e \in QH_{\mathfrak b}(M, \Lambda)$ is an idempotent. 

 The following theorem is the generalization of Theorem 3.1 in \cite{EP03}.
\begin{Thm}[Theorem 16.3 in \cite{FOOO11}]\label{Calabi_q_mor}.
Suppose that there exists a ring isomorphism
\[QH_{\mathfrak b}(M, \Lambda) \cong \Lambda \times Q\]
and $e\in QH_{\mathfrak b}(M, \Lambda)$ is the idempotent corresponding to the unit of the first factor of the right hand side. Then the function 
\[\mu^{\mathfrak b}_{e}: \widetilde{\rm Ham}(M, \omega) \to \mathbb{R}\]
is a homogeneous Calabi quasi-morphism.
\end{Thm}

 From standard properties of spectral invariants (Theorem 7.8 in \cite{FOOO11}), $\mu^{\mathfrak b}_{e}$ has two additional properties (Theorem 14.1 in \cite{FOOO11}):

\begin{enumerate}
\item (Lipschitz continuity) There exists a constant $C \ge 0$ such that for any $\tilde\psi, \tilde\phi \in \widetilde{\rm Ham}(M, \omega)$,
\[|\mu^{\mathfrak b}_{e}(\tilde\psi) -\mu^{\mathfrak b}_{e}(\tilde\phi)| \le C \|\tilde\psi \tilde\phi^{-1}\|.\]

\item (Symplectic invariance) For all $\psi \in {\rm Symp}_{0}(M,\omega)$, 
\[\mu^{\mathfrak b}_{e}(\tilde\phi)=\mu^{\mathfrak b}_{e}(\psi \circ \tilde\phi \circ \psi^{-1}).\]
\end{enumerate}
Here $C \le {\rm vol}(M)$ is easily proved as in Proposition 3.5 of \cite{EP03}.

 On the other hand, {\it symplectic quasi-states} are also constructed by using (bulk deformed) spectral invariants. Let $C^{0}(M)$ be the set of continuous functions on $M$.

\begin{Def}[Section 3 in \cite{EP06}]
\normalfont
A functional $\zeta: C^{0}(M) \to \mathbb{R}$ is called symplectic quasi-state if $\zeta$ satisfies the following:
\begin{enumerate}
\item (Normalization) $\zeta(1)=1$.
\item (Monotonicity) $\zeta(F_{1}) \le \zeta(F_{2})$ for any $F_{1}\le F_{2}$.
\item (Homogeneity) $\zeta(\lambda F)=\lambda \zeta(F)$ for any $\lambda \in \mathbb{R}$.
\item (Strong quasi-additivity) If smooth functions $F$ and $G$ are Poisson commutative: $\{F,G\}=0$, then $\zeta(F+G)=\zeta(F)+\zeta(G)$.
\item (Vanishing) If supp$~ F$ is displaceable, then $\zeta(F)=0$.
\item (Symplectic invariance) $\zeta(F)=\zeta(F \circ \psi)$ for any $\psi \in {Symp}_{0}(M,\omega)$.
\end{enumerate}
\end{Def}

 By using the bulk deformed spectral invariant $\rho^{\mathfrak b}(~\cdot~;e)$, a functional $\zeta^{\mathfrak b}_{e}: C^{\infty}(M) \to \mathbb{R}$ is defined by
\[\zeta^{\mathfrak b}_{e}(H):= - \lim_{m\to +\infty} \frac{\rho^{\mathfrak b}(mH;e)}{m}.\]
This functional $\zeta^{\mathfrak b}_{e}$ extends to a functional on $C^{0}(M)$ as follows. We recall the relation between $\zeta^{\mathfrak b}_{e}$ and $\mu^{\mathfrak b}_{e}$ (see Section 14 \cite{FOOO11}). For any $H\in C^{\infty}([0,1]\times M)$, by the {\it shift property} of spectral invariant, we have

\begin{equation}\label{spectral inv shift}
\rho^{\mathfrak b}(\tilde{\phi}_{H};e) =\rho^{\mathfrak b}(H;e) + \frac{1}{{\rm vol}(M)}{\rm Cal}_{M}(H), 
\end{equation}
where ${\rm Cal}_{M}(H)$ is defined by
\[{\rm Cal}_{M}(H) := \int_{0}^{1}dt \int_{M^{2n}} H_{t} ~\omega^{n}.\]
Since $(\tilde{\phi}_{H})^{m} = \tilde{\phi}_{mH}$ for any autonomous Hamiltonian $H$, the following relation is obtained from (\ref{spectral inv shift})
\[\zeta^{\mathfrak b}_{e}(H)= \frac{1}{{\rm vol}(M)} \left( - \mu^{\mathfrak b}_{e}(\tilde\phi_{H}^{1} )+ {\rm Cal}_{M}(H)\right).\]
By the Lipschitz continuity of $\mu^{\mathfrak b}_{e}$, we can extend $\zeta^{\mathfrak b}_{e}$ to a functional on $C^{0}(M)$. From the same argument in Section 6 in \cite{EP06}, this functional $\zeta^{\mathfrak b}_{e}:C^{0}(M) \to \mathbb{R}$ becomes a symplectic quasi-state if one takes an idempotent $e$ from a field factor of $ QH_{\mathfrak b}(M, \Lambda)$ as in Theorem \ref{Calabi_q_mor}.

 In this paper, we define {\it superheavy subsets} as follows.
\begin{Def}
\normalfont
Let $\zeta$ be a symplectic quasi-state on $(M,\omega)$. A closed subset X $\subset M$ is called $\zeta$-superheavy if for all $H \in C^{0}(M)$
\[\min_{X} H \le \zeta(H) \le \max_{X} H . \]
\end{Def}
 It is immediately proved that any $\zeta$-superheavy subsets must intersect each other and non-displaceable (see \cite{EP09} for details).

\subsection{Brief review of FOOO's results}\label{review_FOOO}

In \cite{FOOO12}, Fukaya-Oh-Ohta-Ono computed the full {\it potential function}, which is a ``generating function of open-closed Gromov-Witten invariant'', of some Lagrangian tori in $S^{2}\times S^{2}$ and they proved superheavyness of these tori in \cite{FOOO11}. In this section, we briefly describe the construction of their superheavy tori.

 Let $F_{2}$(0) be a symplectic toric orbifold whose moment polytope $P$ is given by
\[P := \{(u_{1}, u_{2}) \in \mathbb{R}^{2}\mid 0\le u_{1} \le 2,~0\le u_{2}\le 1 - \frac{1}{2}u_{1}\}.\]
We denote by $\pi: F_{2}(0) \to P$ the moment map, and denote by $L( u)$ a Lagrangian torus fiber over an interior point $ u \in {\rm Int}(P)$. Then $F_{2}(0)$ has one singular point which corresponds to the point $(0,1)$ in $P$. They constructed a symplectic manifold ${\hat F}_{2}(0)$ which is symplectomorphic to $(S^{2} \times S^{2}, \frac{1}{2}\omega_{std} \oplus \frac{1}{2} \omega_{std})$, by replacing a neighborhood of the singularity with a cotangent disk bundle of $S^{2}$ (for details, see Section 4  \cite{FOOO12}). Under the smoothing, Lagrangian torus fiber $L(u)$ is sent to a Lagrangian torus in $S^{2} \times S^{2}$. In particular, we denote by $T_{\tau}$ $(0 < \tau \le  \frac{1}{2})$ this torus corresponding to $L((\tau, 1-\tau)) \subset F_{2}(0)$.

 For these Lagrangian tori $T_{\tau} \subset S^{2}\times S^{2}$, they obtained the following.
\begin{Thm}[Fukaya-Oh-Ohta-Ono \cite{FOOO11}] \label{FOOO12result}
For any $0 < \tau \le 1/2$, there exist an element ${\mathfrak b}(\tau) \in  H^{even}(M, \Lambda_{0})$ and idempotents $e_{\tau}$ and $e^{0}_{\tau}$, each of which is an idempotent of a field factor of $QH_{{\mathfrak b}(\tau)}(S^{2}\times S^{2}; \Lambda)$ such that
\begin{enumerate}[$(1)$]
\item $T_{\tau}$ is $\mu^{{\mathfrak b}(\tau)}_{e_{\tau}}$-superheavy and $T_{\frac{1}{2}}$ is $\mu^{{\mathfrak b}(\tau)}_{e^{0}_{\tau}}$-superheavy. 
\item $S^{1}_{eq}\times S^{1}_{eq}$ is $\mu^{{\mathfrak b}(\tau)}_{e}$-superheavy for any idempotent $e$ of a field factor of $QH_{{\mathfrak b}(\tau)}(S^{2}\times S^{2}; \Lambda)$. In particular, 
\[\psi (T_{\tau}) \cap (S^{1}_{eq}\times S^{1}_{eq}) \neq \emptyset\]
 for any symplectic diffeomorphism $\psi$ on $S^{2}\times S^{2}$. 
\end{enumerate}
\end{Thm}
Here $\mu^{{\mathfrak b}(\tau)}_{e_{\tau}}$ and $\mu^{{\mathfrak b}(\tau)}_{e^{0}_{\tau}}$ denote homogeneous Calabi quasi-morphisms associated to the idempotents $e_{\tau}, e^{0}_{\tau} \in QH_{{\mathfrak b}(\tau)}(S^{2}\times S^{2}; \Lambda)$ respectively (see Theorem \ref{Calabi_q_mor}).

\begin{Rmk}
\normalfont
\begin{enumerate}[$(1)$]
\item In \cite{FOOO11}, $(1)$ is Theorem 23.4 (2), and (2) is Theorem 1.13. 
\item The notion of $\mu^{{\mathfrak b}}_{e}$-superheavy is defined in Definition 18.5 of \cite{FOOO11} and they remark as Remark 18.6 that $\mu^{{\mathfrak b}}_{e}$-superheavyness implies $\zeta^{{\mathfrak b}}_{e}$-superheavyness. In this paper, we need only to use $\zeta^{{\mathfrak b}}_{e}$-superheavyness. 
\item The quasi-morphisms $\mu^{{\mathfrak b}(\tau)}_{e_{\tau}}$ and $\mu^{{\mathfrak b}(\tau)}_{e^{0}_{\tau}}$ descend to homogeneous Calabi quasi-morphisms on ${\rm Ham}(S^{2}\times S^{2})$ as in \cite{EP03}.
\end{enumerate}
\end{Rmk}

Hereafter, we use only above homogeneous Calabi quasi-morphisms \[\mu^{{\mathfrak b} (\tau)}_{e_{\tau}}:{\rm Ham}(S^{2}\times S^{2}) \to \mathbb{R}\] with $0 < \tau < 1/2$ and denote them by $\mu^{\tau}$. 

\subsection{Pullback of the quasi-morphism $\mu^{\tau}$}
To obtain quasi-morphisms on ${\rm Ham}_{c}(B^{2}\times B^{2}, \bar{\omega}_{0})$, we define a conformally symplectic embedding $\Theta_{\delta}: B^{2}\times B^{2} \hookrightarrow S^2\times S^2$ for each Lagrangian submanifold $L_{\delta} \subset B^{2}\times B^{2}$.

 For each $1 / 2 < \delta \le 1$, we define a conformally symplectic embedding $\theta_\delta : (B^{2}, 2  \omega_0) \hookrightarrow (S^2,  \frac{1}{2} \omega_{std}) \cong (\mathbb{C}P^1 , \omega_{FS})$ by
\[\theta_\delta (z) := [ ~\sqrt{ 1- \delta |z|^2} : \sqrt{ \delta} z ~],\]
where we identify the projective space with a unit sphere by using a stereographic projection with respect to $(1, 0, 0) \in S^2 \subset \mathbb{R}^3$ after regarding the plane $\{v=(v_{1}, v_{2}, v_{3}) \in \mathbb{R}^3 \mid v_1 = 0\}$ as the complex plane $\mathbb{C}$. We note that $\theta_{\delta}^{\ast}(\frac{1}{2} \omega_{std}) = \delta \omega_{0}$ and the image of $\theta_\delta$ is $\{v \in S^2 \mid v_1 < 2 \delta -1 \}$. Moreover, by the map $\theta_\delta$, the circle $T_\delta \subset B^{2}$ is mapped onto the equator $S^1_{0} : = \{v \in S^2 \mid v_1 = 0 \}$ and the real form $Re(B^{2})$ is mapped  into the equator $S^1_{eq} := \{v \in \mathbb{R}^{3} \mid v_{3} = 0\} \subset S^2$ . 

 Using this conformally symplectic embedding, we define $\Theta_{\delta}: B^{2}\times B^{2} \hookrightarrow S^2\times S^2$ by 
\begin{equation}\label{Theta_delta}
\Theta_\delta := \theta_\delta \times \theta_\delta:(B^{2}\times B^{2}, \bar{\omega}_{0})  \hookrightarrow (S^2\times S^2, \bar{\omega}_{std})
\end{equation}
where $\bar{\omega}_{std}$ denotes the symplectic structure $\frac{1}{2} \omega_{std} \oplus \frac{1}{2} \omega_{std}$ on $S^{2}\times S^{2}$. This is a conformally symplectic embedding for each $1/2 < \delta \le 1$. Indeed, it is obvious 
\[{\Theta_\delta}^{\ast} \bar{\omega}_{std} = \delta \bar{\omega}_{0}.\]

 For a time-dependent Hamiltonian $F$ on $B^{2} \times B^{2}$, we define a Hamiltonian $F \circ \Theta_\delta^{-1}$ on $S^{2}\times S^{2}$ by
\[F \circ \Theta_\delta^{-1} (x):= 
\begin{cases}
F (t, \Theta_\delta^{-1} (x)) & (x \in {\rm Im}(\Theta_\delta)) \\
0 & (x \notin {\rm Im}(\Theta_\delta)).
\end{cases}\]
Since $\Theta_{\delta}$ is a conformally symplectic embedding, we obtain \[\phi^1_{\delta F\circ \Theta^{-1}_{\delta}}= \Theta_\delta \phi_{F}^1 \Theta_\delta^{-1}.\] Thus, $\Theta_\delta \phi \Theta_\delta^{-1}$ is a Hamiltonian diffeomorphism on $S^{2} \times S^{2}$  for any $\phi \in {\rm Ham}_{c}(B^{2}\times B^{2}, \bar{\omega}_{0})$. 

 We define a family of quasi-morphisms $\mu^{\tau}_{\delta}: {\rm Ham}_{c}(B^{2}\times B^{2}, \bar{\omega}_{0}) \to \mathbb{R}$ by
 
\begin{equation}\label{def_pullbacked_quasimor}
\mu^\tau_\delta(\phi) := \frac{\delta^{-1}}{{\rm vol}(S^2\times S^2)}\left( - \mu^\tau(\Theta_\delta \phi \Theta_\delta^{-1}) +{\rm Cal}_{ \Theta_{\delta}(B^{2}\times B^{2})}(\Theta_\delta \phi \Theta_\delta^{-1})\right),
\end{equation}
where $\mu^{\tau}$ are Fukaya-Oh-Ohta-Ono's quasi-morphisms in Section \ref{review_FOOO} and ${\rm Cal}_{ \Theta_{\delta}(B^{2}\times B^{2})}$ is the Calabi morphism on ${\rm Ham}_{\Theta_{\delta}(B^{2}\times B^{2})}(S^{2}\times S^{2},\bar{\omega}_{std})$ in Section \ref{q_mor and q_state}. The symplectic structure $\bar{\omega}_{std}$ is exact on $\Theta_{\delta}(B^{2}\times B^{2})$, hence the right hand side of (\ref{def_pullbacked_quasimor}) does not depend on the choice of the Hamiltonian generating $\phi$. Moreover, by the definition, it turns out that $\mu^{\tau}_{\delta}$ are quasi-morphisms.

 To obtain another expression of $\mu^\tau_{\delta}$, we define $\zeta^\tau: C^{\infty}([0,1] \times S^{2} \times S^{2}) \to \mathbb{R}$ as the following :
\[\zeta^\tau (H) := - \lim_{n \to \infty} \frac{\rho^{{\mathfrak b}(\tau)} ( H^{\# n}; e_\tau) }{n},\]
where we denote by $H_1 \# H_2$ the concatenation of two Hamiltonian $H_{1}$ and $H_{2}$ :
 \[H_1 \# H_2 (t, x):= 
\begin{cases}
\chi'(t) H_1 (\chi(t), x)) & 0 \le t \le 1/2  \\
\chi'(t-1/2) H_2 (\chi(t), x)) & 1/2 \le t \le 1 
\end{cases}\]
for a smooth function $\chi : [0,1/2] \to [0,1]$ with $\chi ' \ge 0$ and $\chi \equiv 0$ near $t = 0$, $\chi \equiv1$ near $t = 1/2$. Note that this definition is independent of the function $\chi$ since the spectral invariant $\rho^{{\mathfrak b}(\tau)}$ has homotopy invariance property. 

 By the definition and (\ref{spectral inv shift}), one can check that 
\begin{equation}\label{relation_qusimor_qusistate}
\zeta^\tau (H) = \frac{1}{{\rm vol}(S^2\times S^2)} \left( - \mu^\tau (\phi_{H}^1) + {\rm Cal}_{S^2 \times S^2} (H) \right)
\end{equation}
for any time-dependent Hamiltonian $H$ and the restriction of $\zeta^\tau$ to autonomous Hamiltonians corresponds to the bulk-deformed quasi-state $\zeta^{{\mathfrak b}(\tau)}_{e_{\tau}}$ which is associated to $\mu^{\tau} = \mu^{{\mathfrak b}(\tau)}_{e_{\tau}}$.

 Therefore, by (\ref{def_pullbacked_quasimor}) and (\ref{relation_qusimor_qusistate}), we obtain the following expression of $\mu^\tau_{\delta}$.
\begin{Lem}\label{expression_with_qusistate}
\begin{equation*}
\mu^\tau_{\delta} (\phi^1_F) = \delta^{-1} \zeta^\tau (\delta F \circ \Theta_{\delta}^{-1}).
\end{equation*}
\end{Lem}

\section{Properties of quasi-morphisms $\mu^\tau_\delta$ on ${\rm Ham_{c}}(B^{2}\times B^{2}, \bar{\omega}_{0})$}
In this section, we prove some properties of the quasi-morphisms $\mu^\tau_\delta$ by following procedures in \cite{Se13}. Since Proposition \ref{vanishing} and Proposition \ref{superheavy} are proved by using only standard properties of Calabi quasi-morphisms, two proofs are the same as in \cite{Se13}. However the proof of Proposition \ref{invariance} depends on some properties of Lagrangian submanifolds and ambient spaces, thus we need to modify the proof slightly for our Lagrangian submanifolds $L_{\delta} \subset B^{2}\times B^{2}$.

\begin{Prop}\label{vanishing}
For any  $0 < \tau < 1/2$ and $1/2 < \delta \le 1$, we have
\begin{enumerate}[$(1)$]
\item $|\mu^\tau_\delta(\phi)| \le C_{\delta} \| \phi \|$,  where $C_{\delta}$ is a positive constant.
\item If a time-dependent Hamiltonian $H_{t}$ on $B^{2} \times B^{2}$ is supported in a displaceable subset for any time $t \in [0,1]$ then we have \[\mu^\tau_\delta (\phi^{1}_{H}) = 0 .\]
\end{enumerate}
\end{Prop}
\begin{proof}
Let $\phi_{F}^{1}$ be an element  in ${\rm Ham}_{c}(B^{2} \times B^{2}, {\bar \omega}_{0})$.
Since the quasi-morphisms $\mu^{\tau}$ have Lipschitz continuity property with respect to the Hofer norm on ${\rm Ham}(S^{2} \times S^{2}, {\bar \omega}_{std})$ and $\Theta_\delta \phi_{F}^{1} \Theta_\delta^{-1} = \phi^{1}_{\delta F \circ \Theta_{\delta}^{-1}}$, we obtain 
\[
|\mu^{\tau}(\Theta_\delta \phi_{F}^{1} \Theta_\delta^{-1})|
\le
 {\rm vol}(S^{2}\times S^{2}) \|\phi_{\delta F \circ  \Theta_\delta^{-1}}^{1}\|.
\]
By the definition of the Hofer norm, it turns out that
\[\|\phi_{\delta F \circ  \Theta_\delta^{-1}}^{1}\| \le \delta \|\phi^{1}_{F}\|. \]
Hence, we have
\[
|\mu^{\tau}(\Theta_\delta \phi_{F}^{1} \Theta_\delta^{-1})|
\le
\delta {\rm vol}(S^{2}\times S^{2})  \|\phi^{1}_{F}\|.
\]
On the other hand, an easily calculation shows that 
\[{\rm Cal}_{ \Theta_{\delta}(B^{2}\times B^{2})}(\Theta_\delta \phi^{1}_{F} \Theta_\delta^{-1}) = \delta^{3} \int_{0}^{1} dt \int_{B^{2} \times B^{2}} F (t, x) ~{\bar{\omega}_{0}}^{2}.\]
As a result, we can obtain the following:
\[
|{\rm Cal}_{ \Theta_{\delta}(B^{2}\times B^{2})}(\Theta_\delta \phi^{1}_{F} \Theta_\delta^{-1})|
\le
\delta^{3} {\rm vol}(B^{2}\times B^{2})  \|\phi^{1}_{F}\|.
\]
Consequently, it turns out that
\begin{eqnarray*}
|\mu^\tau_\delta(\phi)| & \le &
 \frac{\delta^{-1}}{{\rm vol}(S^{2}\times S^{2})}\Bigl(|\mu^{\tau}(\Theta_\delta \phi \Theta_\delta^{-1})| + |{\rm Cal}_{ \Theta_{\delta}(B^{2}\times B^{2})}(\Theta_\delta \phi \Theta_\delta^{-1})| \Bigl) \\ & \le &
 (1+\delta^{2})\|\phi\|.
\end{eqnarray*}
Thus (1) is proved.

 The property (2)  follows immediately from Calabi-property of $\mu^{\tau}$. Indeed, two terms in the definition of $\mu^{\tau}_{\delta}$ are canceled each other.
\end{proof}

Let $X \subset S^{2}\times S^{2}$ be a $\zeta^{{\mathfrak b}(\tau)}_{e_{\tau}}$-superheavy subset. By definition, we have 
\[\min_{X} H \le \zeta^{{\mathfrak b}(\tau)}_{e_{\tau}}(H) \le \max_{X} H \]
for all autonomous Hamiltonians $H$ on $S^{2}\times S^{2}$.
One can obtain the same inequality for $\zeta^\tau: C^{\infty}([0,1] \times S^{2} \times S^{2}) \to \mathbb{R}$  if a closed subset $X\subset S^{2}\times S^{2}$ is  $\zeta^{{\mathfrak b}(\tau)}_{e_{\tau}}$-superheavy. More precisely, for all time-dependent Hamiltonians $H$ on $S^{2}\times S^{2}$, we have
\begin{equation}\label{ineq_asymptotic_quasistate}
\min_{[0,1] \times X} H \le \zeta^{\tau}(H) \le \max_{[0,1]\times X} H . 
\end{equation}
This is easily proved  as mentioned in \cite{Se13} without the detail. Indeed, we can take two autonomous Hamiltonians $H_{\min}, H_{\max}$ for any time-dependent Hamiltonian $H$ such that $H_{\min} \equiv \min_{[0,1] \times X} H$, $H_{\max} \equiv \max_{[0,1] \times X} H$ on $X$ and $H_{\min}\le H \le H_{\max}$ on $S^{2}\times S^{2}$. By applying  the anti\footnote{Fukaya-Oh-Ohta-Ono used different sign conventions from \cite{EP03, EP06, EP09} (see Remark 4.17 in \cite{FOOO11}). }-monotonicity property of $\rho^{{\mathfrak b}(\tau)}$ (i.e. $H \le K \Rightarrow \rho^{{\mathfrak b}(\tau)}(H; e_{\tau}) \ge \rho^{{\mathfrak b}(\tau)}(K; e_{\tau})$, see Theorem 9.1 in \cite{FOOO11}) and the fact $H \le K$ implies $H^{\# n} \le K^{\# n}$ to above Hamiltonians $H_{\min}, H, H_{\max}$, we can obtain (\ref{ineq_asymptotic_quasistate}) immediately.

 From Lemma \ref{expression_with_qusistate} and this inequality (\ref{ineq_asymptotic_quasistate}), we obtain the following.

\begin{Prop}\label{superheavy}
Suppose a closed subset $X \subset S^2 \times S^2$ is $\zeta^{{\mathfrak b}(\tau)}_{e_{\tau}}$-superheavy and $F$ is any compactly supported time-dependent Hamiltonian on the bi-disks $B^{2} \times B^{2}$ such that  $F \circ \Theta_\delta^{-1} \mid_{X} \equiv c$, then 
\[\mu^\tau_\delta (\phi^1_F) =c .\]
\end{Prop}

 Proposition \ref{invariance} is the most important to obtain unboundedness of $(\mathcal{L}(L_{\delta}), d)$. In \cite{Kh09}, Khanevsky proved the similar property  and obtained the unboundedness for the case where the ambient space is two-dimensional open ball. In \cite{Se13}, by a different proof, Seyfaddini also obtained the similar property for $(\mathcal{L}(Re(B^{2n})), d)$. 

\begin{Prop}\label{invariance}
If two Hamiltonian diffeomorphisms $\phi, \psi \in {\rm Ham}_c(B^{2} \times B^{2} , \bar{\omega}_{0})$ satisfy 
\[\phi(L_\delta) = \psi(L_\delta), \]
then we have
\[| \mu^\tau_\delta (\phi) - \mu^\tau_\delta (\psi) | \le \frac{D_{\mu^\tau}}{\delta {\rm vol}(S^{2}\times S^{2})}~~~{\rm for~all}~~\frac{1}{2} < \delta \le 1, ~0 < \tau < \frac{1}{2},\]
where $D_{\mu^\tau}$ denotes the defect of $\mu^\tau$.
\end{Prop}
 We prove this proposition by slightly modifying Seyfaddini's proof.
\begin{proof}
Throughout the proof, we fix $\delta$, $\tau$ with $1/2< \delta \le 1$, $0 < \tau < 1/2$, respectively. From the definition of $\mu^\tau_\delta$ and its homogeneity we obtain that
\begin{equation*}
	\begin{split}
 &|\mu^\tau_\delta (\phi^{-1}\psi) +\mu^\tau_\delta (\phi) - \mu^\tau_\delta (\psi)| \\
 &= |\mu^\tau_\delta (\phi^{-1}\psi) -\mu^\tau_\delta (\phi^{-1}) - \mu^\tau_\delta (\psi)| \\
 &=\frac{1}{\delta {\rm vol}(S^{2}\times S^{2})}|\mu^\tau(\Theta_\delta \phi^{-1}\psi \Theta^{-1}_\delta) - \mu^\tau (\Theta_\delta \phi^{-1} \Theta_\delta^{-1}) - \mu^\tau (\Theta_\delta \psi \Theta_\delta^{-1})|\\
 & \le \frac{D_{\mu^\tau}}{\delta {\rm vol}(S^{2}\times S^{2})}.
	\end{split}
\end{equation*}
Consequently, it is sufficient to prove the proposition that $\mu^\tau_{\delta}(\phi)$ vanishes for Hamiltonian diffeomorphisms $\phi$ satisfying $\phi (L_{\delta}) = L_{\delta}$.

 Now we take any Hamiltonian $F \in C^\infty_c([0,1] \times (B^{2} \times B^{2}))$ and assume the Hamiltonian diffeomorphism $\phi^{1}_{F}$ preserves the 
Lagrangian submanifold $L_{\delta}$.

 For $0 < s \le 1$, we define a diffeomorphism $a_s : B^{2} \times B^{2}(s) \to B^{2} \times B^{2}$ by
\[a_{s}(z_1,z_2):=(z_1, \frac{z_2}{s}).\]
Using this map, we define a compactly supported symplectic diffeomorphism $\psi_{s}$ for each $0<s\le1$:
\[\psi_s :=
\begin{cases}
a_s^{-1} \phi^1_{F} a_s & |z_2| \le s  \\
id & |z_2| \ge s
\end{cases}
.\]
As compactly supported cohomology group $H^{1}_{c}(B^{2} \times B^{2} ; \mathbb{R})=0$ and $\bar{\omega}_{0}$ is exact on $B^{2} \times B^{2}$, any isotopy of compactly supported Symplectic diffeomorphisms on  $(B^{2} \times B^{2} , \bar{\omega}_{0})$ is a compactly supported Hamiltonian isotopy. Thus, for each $0<s\le 1$, we can take a time-dependent Hamiltonian $F^{s} \in C_{c}^{\infty}([0,1] \times B^{2}\times B^{2})$ such that $\psi_{s} = \phi^{1}_{F^{s}}$.

 This Hamiltonian diffeomorphisms $\psi_{s}$ have the following properties:
\begin{enumerate}[(1)]
\item $\psi_{1} = \phi^{1}_{F^{1}}=\phi^{1}_{F}$,
\item $\psi_{s}$ preserves $L_{\delta}$ for each $0 < s \le 1$,
\item There exists a compact subset $K_{s}$ in $B^{2}$ such that $F^{s}$ is supported in $K_{s} \times B^{2}(s) \subset B^{2} \times B^{2}$ for each $0 < s \le 1$.
\end{enumerate}

 Hereafter we fix sufficiently small $\epsilon > 0$ such that $K_{\epsilon} \times B^{2}(\epsilon)$ is displaceable inside the bi-disks $B^{2} \times B^{2}$. By Proposition \ref{vanishing} (2), it follows that
\begin{equation}\label{vanishing_by_supp}
\mu^{\tau}_{\delta}(\psi_{\epsilon}) = 0.
\end{equation}

 We take a time-dependent Hamiltonian $H \in C_{c}^{\infty}([0,1] \times B^{2}\times B^{2})$ so that $\phi^t_H := \psi^{-1}_\epsilon \psi_{t(1-\epsilon) + \epsilon}$ for $0 \le t \le 1$. In particular, we have the time-one map $\phi^1_H = \psi^{-1}_{\epsilon} \phi^{1}_{F}$ by the above property (1). 

 We note that Hamiltonian vector field $X_{H_{t}}$ is tangent to the Lagrangian submanifold $L_{\delta}$ since $\phi^t_H$ preserves $L_{\delta}$. Consequently, for each $t \in [0,1]$, $H_{t}=H(t,\cdot)$ is constant on $L_{\delta}$. Because of this and non-compactness of $L_{\delta}$, the restriction of $H_{t}$ to $L_{\delta}$ is $0$ for all $t \in [0,1]$. Since $L_{\delta} = T_{\delta} \times Re(B^{2})$ is mapped into $S^{1}_{0} \times S^{1}_{eq}$ by $\Theta_{\delta}$, hence $H \circ \Theta^{-1}_{\delta}$ vanishes on a torus $S^{1}_{0} \times S^{1}_{eq}$. On the other hand $S^{1}_{0} \times S^{1}_{eq}$ is $\zeta^{{\mathfrak b}(\tau)}_{e_{\tau}}$-superheavy by Fukaya-Oh-Ohta-Ono's result (Theorem \ref{FOOO12result}), therefore we have 
\begin{equation}\label{vanishing_by_superheavy}
\mu^{\tau}_{\delta} (\phi^{1}_{H}) = 0.
\end{equation}
Here we used Proposition \ref{superheavy}.

 As a consequence of these two equalities (\ref{vanishing_by_supp}), (\ref{vanishing_by_superheavy}) and quasi-additivity of $\mu^{\tau}_{\delta}$, it follows that
\[|\mu^{\tau}_{\delta}(\phi^{1}_{F})| = |\mu^{\tau}_{\delta}(\phi^{1}_{F}) - \mu^{\tau}_{\delta}(\psi_{\epsilon})-\mu^{\tau}_{\delta}(\phi^{1}_{H}) | \le  \frac{D_{\mu^\tau}}{\delta {\rm vol}(S^{2}\times S^{2})}.\]
Because $(\phi^{1}_{F})^{n}$ preserves $L_{\delta}$  for any $n \in \mathbb{N}$, we can apply the same argument to $(\phi^{1}_{F})^{n}$ and obtain $|\mu^{\tau}_{\delta}((\phi^{1}_{F})^{n})| \le \delta^{-1} {\rm vol}(S^{2}\times S^{2})^{-1}  D_{\mu^\tau}$.
Since $\mu^{\tau}_{\delta}$ is a homogeneous quasi-morphism, we have
\[\mu^{\tau}_{\delta}(\phi^{1}_{F}) = 0.\]

\end{proof}

 By applying Proposition \ref{vanishing} (1) and Proposition \ref{invariance}, we obtain the following.

\begin{Prop}\label{estimate_by_quasimor}
For any $\phi \in {\rm Ham}_{c}(B^{2}\times B^{2}, \bar{\omega}_{0})$ and any $\frac{1}{2} < \delta \le 1, ~0 < \tau < \frac{1}{2}$, the following inequality holds.
\[\frac{\mu^{\tau}_{\delta}(\phi) - \delta^{-1} {\rm vol}(S^{2}\times S^{2})^{-1}D_{\mu^{\tau}}}{C_{\delta}} \le d(L_{\delta}, \phi(L_{\delta})),\]
where $D_{\mu^{\tau}}$ is as above.
\end{Prop}
\begin{proof}
We take any $\psi \in {\rm Ham}_{c}(B^{2}\times B^{2}, \bar{\omega}_{0})$ satisfying $\phi(L_{\delta}) = \psi(L_{\delta})$. From Proposition \ref{invariance}, we obtain the following inequality.
\[|\mu^{\tau}_{\delta}(\phi) - \mu^{\tau}_{\delta}(\psi)| \le  \frac{D_{\mu^\tau}}{\delta {\rm vol}(S^{2}\times S^{2})}.\]
By using Proposition \ref{vanishing} (1), we have 
\[|\mu^{\tau}_{\delta}(\phi) | -  \frac{D_{\mu^\tau}}{\delta {\rm vol}(S^{2}\times S^{2})} \le |\mu^{\tau}_{\delta}(\psi)| \le C_{\delta}\|\psi\|.\]
Therefore, by definition of the metric $d$, we obtain the following inequality:
\[|\mu^{\tau}_{\delta}(\phi) | -  \frac{D_{\mu^\tau}}{\delta {\rm vol}(S^{2}\times S^{2})}  \le C_{\delta}\cdot d(L_{\delta}, \psi(L_{\delta})).\]
\end{proof}

\section{Construction of $\Phi_\delta :C^{\infty}_c((0,1)) \to \mathcal{L}(L_\delta)$}
\subsection{Locations of FOOO's superheavy tori}
To construct a mapping $\Phi_\delta :C^{\infty}_c((0,1)) \to \mathcal{L}(L_\delta)$ in Theorem \ref{main2}, we describe the locations of Fukaya-Oh-Ohta-Ono's Lagrangian superheavy tori by following Oakley-Usher's result. Let us recall their description. In \cite{OU13}, they constructed a symplectic toric orbifold $\mathcal O$ which is isomorphic to $F_{2}(0)$ as symplectic toric orbifolds by gluing $S^{2}\times S^{2} \setminus \bar \Delta $ to $B^{4} / \{\pm 1\}$. Here $\bar \Delta$ denotes anti-diagonal of $S^{2}\times S^{2}$ and $B^{4}$ is a four dimensional open ball. The moment map $\pi : \mathcal{O} \to \mathbb{R}^{2}$, which has the same moment polytope $P$ of $F_{2}(0)$ in Section \ref{review_FOOO}, is expressed on $S^{2}\times S^{2} \setminus \bar \Delta $ by
\[\pi(v,w) = \Bigl(\frac{1}{2}|v + w| + \frac{1}{2}(v + w) \cdot e_{1},~ 1-\frac{1}{2}|v + w|\Bigr) \in \mathbb{R}^{2}\]
for $(v,w) \in S^{2}\times S^{2} \setminus \bar \Delta$ and $e_{1} := (1, 0, 0)$. Therefore one can consider a torus fiber $L(u) \subset F_{2}(0)$ as $\pi^{-1}(u) \subset S^{2}\times S^{2} \setminus \bar \Delta$ for any interior point $u$ in the moment polytope.

 By replacing $B^{4} / \{\pm 1\}$ by the unit disk cotangent bundle $D^{\ast}_{1}S^{2}$, they obtained a smoothing $\Pi :{ \hat {\mathcal O}} \to \mathcal O $ which maps the zero-section of $D^{\ast}_{1}S^{2}$ to the singularity of $\mathcal O$ and whose restriction to $S^{2}\times S^{2} \setminus \bar \Delta $ is the identity mapping. Moreover they gave an explicit symplectic morphism ${ \hat {\mathcal O}} \xrightarrow{\sim} S^{2} \times S^{2}$ which is the identity mapping on $S^{2}\times S^{2} \setminus \bar \Delta $. Hence above tori $\pi^{-1}(u)$ are invariant under the smoothing and the symplectic morphism ${ \hat {\mathcal O}} \xrightarrow{\sim} S^{2} \times S^{2}$. 

 Using this construction, Oakley-Usher  proved that the Entov-Polterovich's exotic monotone torus in \cite{EP09} is Hamiltonian isotopic to the Fukaya-Oh-Ohta-Ono's torus over $(1/2, 1/2)$ (for details, see the proof of Proposition 2.1 \cite{OU13}).
\begin{Prop}[Oakley-Usher \cite{OU13}]
Fukaya-Oh-Ohta-Ono's superheavy Lagrangian tori $T_{\tau}$ can be expressed as  
\[T_\tau = \left\{(v,w) \in S^2 \times S^2 \mid \frac{1}{2}|v + w| + \frac{1}{2}(v + w) \cdot e_{1} = \tau, ~ 1-\frac{1}{2}|v + w| = 1- \tau \right\},\]
where the parameter $\tau$ is in $(0, 1/2]$. In particular, the Lagrangian torus $T_{1/2}$ is Entov-Polterovich's exotic monotone torus.
\end{Prop}

 The following corollary is proved by an easily calculation.
\begin{Cor}
The image of $i$-th projection ${\rm pr}_i : S^2 \times S^2 \to S^2$ $(i=1,2)$ is 
\begin{equation}
{\rm pr}_{i} (T_\tau) = \left\{ v \in S^2 \mid |v \cdot e_1| \le \sqrt{1- \tau^2}\right\},
\end{equation}
where $\tau$ is $0 < \tau \le 1/2$.
\end{Cor}

 By this corollary and the definition of the conformally symplectic embedding $\Theta_\delta :B^{2}\times B^{2} \hookrightarrow S^2\times S^2$. We have the following.
\begin{Cor}\label{image_cor}
For any $(2 + \sqrt{3}) /4 < \delta \le 1$ there exists a sufficiently small $\varepsilon_\delta > 0$ such that
\[\bigcup_{\tau \in I_\delta} T_\tau  \subset  \Theta_\delta (B^{2} \times B^{2}),~~~~~~I_\delta := [1/2 - \varepsilon_\delta , 1/2].\]
\end{Cor}

\begin{Rmk}\label{delta_condition}
\normalfont
The condition $(2 + \sqrt{3}) /4 < \delta \le 1$ in Theorem \ref{main2} guarantees that the image of $\Theta_\delta$ contains a continuous subfamily of superheavy tori $T_{\tau} \subset S^{2} \times S^{2}$ as in Corollary \ref{image_cor}. However, for any $1/2 < \delta \le 1$, it is likely that there exist $\phi_{\delta} \in {\rm Ham}(S^{2}\times S^{2})$ such that the image of $\Theta_\delta$ contains $\cup_{\tau \in I'_\delta} \phi_{\delta}(T_{\tau})$ for some open interval $I'_\delta \subset (0,1/2]$. In this case, we can show Theorem \ref{main2} under the weaker assumption $1/2 < \delta \le 1$.
\end{Rmk}

\subsection{Construction of $\Phi_\delta$}
We fix $\delta$ with $(2 + \sqrt{3}) /4 < \delta \le 1$ and consider the interval $I_{\delta} = [1/2 - \varepsilon_{\delta}, 1/2]$ in Corollary \ref{image_cor}. We take a segment $J_{\delta}$ in the moment polytope $P = \pi (\mathcal{O}) \subset \mathbb{R}^{2}$ defined by
\[J_{\delta} := \{(\tau, 1- \tau) \mid \tau \in {\rm Int}(I_{\delta})\} \subset {\rm Int}(P).\]
We denote by $B^{2}(u_{0}; \sqrt2 \varepsilon_{\delta})$ the open disk of which center is $u_{0}:=(1/2,1/2) \in {\rm Int}(P)$ and radius is $\sqrt2 \varepsilon_{\delta}$. We may take and fix a sufficiently small $\varepsilon_{\delta} > 0$ so that the open disk $B^{2}(u_{0}; \sqrt2 \varepsilon_{\delta})$ is contained in $P$ and moreover the inverse image of $B^{2}(u_{0}; \sqrt2 \varepsilon_{\delta})$ under $\tilde{\pi} := \pi \circ \Pi : \hat{O} \to P$ is contained in the image of $\Theta_{\delta}:B^{2} \times B^{2} \to S^{2} \times S^{2}$.

 We identify  $J_{\delta}$ with an open interval $(0,1)$ and will define a map $\Phi_{\delta}$ on $C^{\infty}_{c}(J_{\delta})$. First, we extend a function $f \in C^{\infty}_{c}(J_{\delta})$ to the function $f_{B^2}$ on the open disk $B^{2}(u_{0}; \sqrt2 \varepsilon_{\delta})$ which is constant along the circle centered at $u_{0}$. More explicitly, we define $f_{B^2}: B^{2}(u_{0}; \sqrt2 \varepsilon_{\delta}) \to \mathbb{R}$ by

\[f_{B^2}(u) : = f \large((~|u-u_{0}| / \sqrt{2}, 1 - |u-u_{0}| /\sqrt{2} ~)\large),~~ u \in B^{2}(u_{0}; \sqrt2 \varepsilon_{\delta}) \subset {\rm Int}(P).\]

 We define $\tilde f \in C^{\infty}_{c}( B^{2} \times B^{2}) $ for $f \in C^{\infty}_{c}(J_{\delta})$ as the pull-back: 
\begin{equation}\label{f_tilde_def}
\tilde f :=  \Theta_{\delta}^{\ast}\tilde\pi^{\ast}f_{B^2}.
\end{equation}

By the construction, the restriction of $\tilde f$ on $\Theta_{\delta}^{-1}(T_{\tau})$ is constantly equal to $f(\tau)$ for all $1/2 - \varepsilon_{\delta}< \tau < 1/2$.

\begin{Def}\label{Psi_delta_def}
\normalfont
For any $(2 + \sqrt{3})/4 < \delta \le1$, we define $\Phi_{\delta} : C^{\infty}_{c}((0, 1)) \to \mathcal{L}(L_{\delta})$ by the following expression:
\[\Phi_{\delta} (f) := \phi^{1}_{\tilde{f}}(L_{\delta}), \]
where we regard $f$ as an element in $C^{\infty}_{c}(J_{\delta})$.
\end{Def}

 For the proof of Theorem \ref{main2}, we prove the next lemma.
\begin{Lem}\label{lem__norm}
For any $f,g \in C^{\infty}_c((1/2 - \varepsilon_{\delta},1/2))$ there exists a constant $1/2 - \varepsilon_{\delta} < \tau' < 1/2$ such that 
\[|\mu^{\tau'}_\delta(\phi^1_{\tilde{f}-\tilde{g}})| = \|f - g \|_{\infty},\]
where $\delta$ is $(2 + \sqrt{3})/4 < \delta \le1$.
\end{Lem}
\begin{proof}
For any $f,g \in C^{\infty}_c((1/2 - \varepsilon_{\delta}, 1/2))$, there exists $\tau' \in (1/2 - \varepsilon_{\delta}, 1/2)$ such that
\[\|f-g\|_{\infty} = \max |f(x)-g(x)| = |f(\tau') - g(\tau')|.\]
Thus $\mu^{\tau'}_\delta(\phi^1_{\tilde{f}-\tilde{g}})$ is equal to $\|f-g\|_{\infty}$ because of (\ref{f_tilde_def}) and Proposition \ref{superheavy}.
\end{proof}
\section{Proof of Theorem \ref{main1} and Theorem \ref{main2}.}

\begin{proof}[proof of Theorem \ref{main1}]
For all $1/2 < \delta \le 1$, the image of $\Theta_{\delta}$ contains the torus $S^{1}_{0}\times S^{1}_{0} \subset (S^2\times S^2, \bar{\omega}_{std})$. If we take a Hamiltonian $H \in C_{c}^{\infty}(B^{2}\times B^{2})$ for any $h \in \mathbb{R}$ such that $H \equiv h$ on the torus $\Theta_{\delta}^{-1}(S^{1}_{0}\times S^{1}_{0})$, then we have from Proposition \ref{superheavy} and $\zeta^{{\mathfrak b}(\tau)}_{e_{\tau}}$-superheavyness of  $S^{1}_{0}\times S^{1}_{0}$
\[\mu_{\delta}^{\tau}(\phi^{1}_{H}) = h,\]
where we fix any $\tau \in (0, \frac{1}{2})$. By applying Proposition \ref{estimate_by_quasimor}, we obtain
\[\frac{h- \delta^{-1} {\rm vol}(S^{2}\times S^{2})^{-1}D_{\mu^{\tau}}}{C_{\delta}} \le d(L_{\delta}, \phi(L_{\delta})).\]
Since $h$ is an arbitrary constant, Theorem \ref{main1} is proved.
\end{proof}

 Theorem \ref{main1} is proved by using a single quasi-morphism $\mu_{\delta}^{\tau}$ on ${\rm Ham}_{c}(B^{2}\times B^{2}, \bar{\omega}_{0})$.

 On the other hand, to prove Theorem \ref{main2}, it is necessary that the image $\Theta_{\delta}(B^{2}\times B^{2})$ contains a continuous subfamily of superheavy tori $\phi_{\delta}(T_{\tau}) \subset S^{2} \times S^{2}$ for some $\phi_{\delta} \in {\rm Ham}(S^{2}\times S^{2})$ as mentioned in Remark \ref{delta_condition}.

 In this paper, we consider the case $\phi_{\delta} = id$. Then we need to use the parameter $\delta$ of our Lagrangian submanifolds $L_{\delta}$ with $(2 + \sqrt{3})/4 < \delta \le 1$ as in Corollary \ref{image_cor}.

\begin{proof}[proof of Theorem \ref{main2}]
First, we will prove the left-hand side inequality. For any $f, g \in C^{\infty}_c((1/2 - \varepsilon_{\delta}, 1/2))$, we have ${\tilde f}, {\tilde g} \in C^{\infty}_{c}(B^{2}\times B^{2})$ defined by (\ref{f_tilde_def}). Then we apply Proposition \ref{estimate_by_quasimor} to $\phi^{-1}_{\tilde g} \circ \phi^{1}_{\tilde f} \in {\rm Ham}_{c}(B^{2}\times B^{2}, {\bar \omega}_{0})$ to obtain
\begin{equation}\label{ineq_section3}
\frac{|\mu^{\tau}_{\delta}(\phi^{-1}_{\tilde g} \circ \phi^{1}_{\tilde f})|-\delta^{-1} {\rm vol}(S^{2}\times S^{2})^{-1}D_{\mu^{\tau}}}{C_{\delta}} \le d(L_{\delta}, \phi^{-1}_{\tilde g} \circ \phi^{1}_{\tilde f}(L_{\delta})),
\end{equation}
where $\phi^{-1}_{\tilde g}$ is the inverse of $\phi^{1}_{\tilde g}$.
By the construction of autonomous Hamiltonians ${\tilde f}$, ${\tilde g}$ in (\ref{f_tilde_def}), we find that the Poisson bracket $\{{\tilde f}, {\tilde g}\}_{\bar{\omega}_{0}}$ vanishes. Thus we have
\[\phi^{-1}_{\tilde g} \circ \phi^{1}_{\tilde f} = \phi^{1}_{{\tilde f}-{\tilde g}}.\]
Therefore the inequality (\ref{ineq_section3}) becomes
\[
\frac{|\mu^{\tau}_{\delta}(\phi^{1}_{{\tilde f}-{\tilde g}})|-\delta^{-1} {\rm vol}(S^{2}\times S^{2})^{-1}D_{\mu^{\tau}}}{C_{\delta}} \le d(\phi^{1}_{\tilde g}(L_{\delta}), \phi^{1}_{\tilde f}(L_{\delta})).
\]
By Lemma \ref{lem__norm}, we obtain the following inequality:
\[
\frac{\|f - g\|_{\infty}-\delta^{-1} {\rm vol}(S^{2}\times S^{2})^{-1}D_{\mu^{\tau'}}}{C_{\delta}} \le d(\Phi_{\delta}(f), \Phi_{\delta}(g)),
\]
where the constant $\tau'$ depends on $f$ and $g$. We prove the following lemma in Section \ref{appendix}.

\begin{Lem}\label{defect}
For any bulk-deformation parameter $\tau \in (0, 1/2)$, the defect $D_{\mu^{\tau}}$ of quasi-morphisms $\mu^{\tau}$ satisfies
\[D_{\mu^{\tau}} \le 12.\]
\end{Lem}

 Therefore, we obtain  the left-hand side inequality by putting $D_{\delta}:= \delta^{-1} {\rm vol}(S^{2}\times S^{2})^{-1}\cdot \sup_{\tau} D_{\mu^{\tau}}$.

 The right-hand side inequality is proved immediately. Indeed, we can estimate as the following:
\begin{equation*}
	\begin{split}
	d(\Phi_{\delta}(f), \Phi_{\delta}(g)) = d(L_{\delta}, \phi^{-1}_{\tilde g} \phi^{1}_{\tilde f}(L_{\delta}))
	&\le  \|{\tilde f}-{\tilde g}\|\\
	&= \|f-g\|.
	\end{split}
\end{equation*}
This completes the proof of Theorem \ref{main2}.
\end{proof}

\section{Finiteness of $D_{\mu^{\tau}}$}\label{appendix}
The estimate in Lemma \ref{defect} can be obtained by almost the same calculation of Proposition 21.7 in \cite{FOOO11}. For this reason, we only sketch the outline of the calculation and use the same notation used in \cite{FOOO11}.

\begin{proof}[proof of Lemma \ref{defect}]

 From Remark 16.8 in \cite{FOOO11}, upper bounds of defects $D_{\mu^{\tau}}$ can be taken to  be $-12{\mathfrak v}_{T}(e_{\tau})$, where ${\mathfrak v}_{T}$ is a valuation of bulk-deformed quantum cohomology $QH_{{\mathfrak b}(\tau)}(S^{2}\times S^{2}; \Lambda)$. The proof of Theorem \ref{FOOO12result} (Theorem 23.4 \cite{FOOO11}) implies that the idempotent $e_{\tau} \in QH_{{\mathfrak b}(\tau)}(S^{2}\times S^{2}; \Lambda)$ can be taken from one of four idempotents in $QH_{{\mathfrak b}(\tau)}(S^{2}\times S^{2}; \Lambda)$ which decompose quantum cohomology as follows:
 \[ QH_{{\mathfrak b}(\tau)}(S^{2}\times S^{2}; \Lambda) = \bigoplus_{(\epsilon_{1},\epsilon_{2}) = (\pm 1, \pm 1)} \Lambda \cdot e^{\tau}_{\epsilon_{1},\epsilon_{2}}. \]
Here the quantum product in $QH_{{\mathfrak b}(\tau)}(S^{2}\times S^{2})$ respects this splitting (i.e. it is semi-simple). 

 Hence, to prove Lemma \ref{defect}, we only have to estimate the maximum valuation of $e^{\tau}_{\epsilon_{1},\epsilon_{2}}$. For this purpose, we regard $S^{2} \times S^{2}$ as the symplectic toric manifold with the moment polytope:
\[P = \{ u=(u_1, u_2) \in \mathbb{R}^2 ~|~ l_i( u) \ge 0 ,~ i =1, \dots ,4\},\]
where
\[
l_1 = u_1, ~~ l_2 = u_2 ,~~ l_3 = - u_1 +1,~~ l_4 =  - u_2 + 1.
\]
We denote by $\partial_{i}P := \{l_{i}(u) = 0\}$ each facets of $P$ and put $D_{i} := \pi^{-1}(\partial_{i}P)$, where $\pi: S^{2} \times S^{2} \to P \subset \mathbb{R}^{2}$ is the moment map. In the following, we fix 
 \[e_{0} := PD[S^{2}\times S^{2}],~e_{1}:=PD[D_{1}],~e_{2}:=PD[D_{2}],~e_{3}:=PD[D_{1}\cap D_{2}]\]
as basis of $H^{\ast}(S^{2}\times S^{2}; \mathbb{C})$ and denote by $L(u_{0})$ the Lagrangian torus fiber over $(1/2,1/2) \in P$.
 
 The element ${\mathfrak b}(\tau)$ in Theorem \ref{FOOO12result} is defined by
\begin{equation}
 {\mathfrak b}(\tau) := a PD[D_{1}] + a PD[D_{2}],~a := T^{\frac{1}{2}-\tau}.
\end{equation}
In our case, since $S^{2} \times S^{2}$ is Fano, the potential function $\mathfrak{PO}_{{\mathfrak b}(\tau)}$ is determined in terms of the moment polytope data. Hence we obtain the following expression as in the proof of Theorem 23.4 \cite{FOOO11}
 
\[
\mathfrak{PO}_{{\mathfrak b}(\tau)} = \mathrm{e}^{a}y_{1} + \mathrm{e}^{-a}y_{2} + y_{1}^{-1}T + y_{2}^{-1}T,
\]
where $y_{1}, \dots ,y_{4}$ are formal variables and $\mathrm{e}^{a} := \sum_{n=0}^{\infty} a^{n}/n! \in \Lambda_{0}$  (see Section 3 in \cite{FOOO11a} and Section 20.4 in \cite{FOOO11} for the definition of potential functions for toric fibers).¡¡

 By Proposition 1.2.16 in \cite{FOOO10}, the {\it Jacobian ring} ${\rm Jac}(\mathfrak{PO}_{{\mathfrak b}(\tau)}; \Lambda)$ of the potential function $\mathfrak{PO}_{{\mathfrak b}(\tau)}$, which is defined as a certain quotient ring of the Laurent polynomial $\Lambda [y_{1}, \dots, y_{4}, y^{-1}_{1}, \dots, y^{-1}_{4}]$ for our case, is decomposed as follows:
\[
{\rm Jac}(\mathfrak{PO}_{{\mathfrak b}(\tau)}; \Lambda) =  \bigoplus_{(\epsilon_{1},\epsilon_{2}) = (\pm 1, \pm 1)} \Lambda \cdot 1^{\tau}_{\epsilon_{1},\epsilon_{2}},
\]
where $1^{\tau}_{\epsilon_{1},\epsilon_{2}}$ is the unit on each component. More explicitly, we have
\[
1^{\tau}_{\epsilon_{1},\epsilon_{2}} = \frac{1}{4}\Large[1 + \epsilon_{1}\mathrm{e}^{\frac{a}{2}} y_{1}T^{-\frac{1}{2}} + \epsilon_{2} \mathrm{e}^{-\frac{a}{2}}y_{2}T^{-1/2} + \epsilon_{1}\epsilon_{2} y_{1}y_{2}T^{-1} \Large].
\]

 We denote by $e^{\tau}_{\epsilon_{1},\epsilon_{2}}$ the idempotent of $QH_{{\mathfrak b}(\tau)}(S^{2}\times S^{2}; \Lambda)$ which corresponds to $1^{\tau}_{\epsilon_{1},\epsilon_{2}}$ under the {\it Kodaira-Spencer map}:
\[\mathfrak{ks}_{{\mathfrak b}(\tau)}: QH_{{\mathfrak b}(\tau)}(S^{2}\times S^{2}; \Lambda) \to {\rm Jac}(\mathfrak{PO}_{{\mathfrak b}(\tau)}; \Lambda),\]
 which is a ring isomorphism (see Theorem 20.18 in \cite{FOOO11}). The same calculation as in Remark 1.3.1 \cite{FOOO10} shows that the Kodaira-Spencer map $\mathfrak{ks}_{{\mathfrak b}(\tau)}$ maps the basis of $QH_{{\mathfrak b}(\tau)}(S^{2}\times S^{2}; \Lambda)$ to the following:
\[\mathfrak{ks}_{{\mathfrak b}(\tau)}(e_{0}) = [1],~~    \mathfrak{ks}_{{\mathfrak b}(\tau)}(e_{1}) = [\mathrm{e}^{a}y_{1}], ~~\mathfrak{ks}_{{\mathfrak b}(\tau)}(e_{2}) = [\mathrm{e}^{-a}y_{2}],~~  \mathfrak{ks}_{{\mathfrak b}(\tau)}(e_{3})=[q y_{1}y_{2}].\]
Here $q\in \mathbb{Q}$ is defined as follows (see Definition 6.7 in \cite{FOOO11a}). Let $\beta_{1}+\beta_{2}$ be the element of $H_{2}(S^{2}\times S^{2},L(u_{0}); \mathbb{Z})$ satisfies
\[(\beta_{1}+\beta_{2}) \cap D_{i} = 1~(i=1,2)\] 
with Maslov index $\mu_{L}(\beta_{1}+\beta_{2})=4$ and
 \[q := ev_{0 \ast}[\mathcal{M}^{main}_{1;1}(L(u_{0}),\beta_{1}+\beta_{2};e_{3})]\cap L(u_{0}),\] 
where we denote by $\mathcal{M}^{main}_{1;1}(L(u_{0}),\beta_{1}+\beta_{2};e_{3})$ the moduli space of genus zero bordered stable maps in class $\beta_{1}+\beta_{2}$ with one boundary point and one interior point whose image lies in $D_{1}\cap D_{2}$ (see Section 6 of \cite{FOOO11a} for the precise definition of the moduli space) .

 The classification theorem of holomorphic disks in \cite{CO06} implies $q=\pm 1$ immediately.

 By comparing $e^{\tau}_{\epsilon_{1},\epsilon_{2}}$ with $1^{\tau}_{\epsilon_{1},\epsilon_{2}}$, we can obtain for $(\epsilon_{1},\epsilon_{2})=(\pm 1, \pm 1)$,
 
\[e^{\tau}_{\epsilon_{1},\epsilon_{2}} = \frac{1}{4} \big(e_{0} + \epsilon_{1}\mathrm{e}^{-\frac{a}{2}}T^{-\frac{1}{2}}\cdot e_{1} + \epsilon_{2}\mathrm{e}^{\frac{a}{2}}T^{-\frac{1}{2}}\cdot e_{2} + \epsilon_{1}\epsilon_{2}q^{-1}T^{-1} \cdot e_{3}\big).\]

 Since $a = T^{\frac{1}{2}-\tau}$ and $ 0 < \tau < 1/2$, we obtain ${\mathfrak v}_{T}(e^{\tau}_{\epsilon_{1},\epsilon_{2}}) = -1$. This implies Lemma \ref{defect}.
\end{proof}

\end{document}